\documentclass[american,english]{article}
\usepackage[T1]{fontenc}
\usepackage[latin9]{inputenc}
\usepackage{array}
\usepackage{textcomp}
\usepackage{amsmath}
\usepackage{amsthm}
\usepackage{amssymb}
\PassOptionsToPackage{normalem}{ulem}
\usepackage{ulem}

\makeatletter

\providecommand{\tabularnewline}{\\}

\numberwithin{equation}{section}
\numberwithin{figure}{section}
\theoremstyle{plain}
\newtheorem{thm}{\protect\theoremname}[section]
\theoremstyle{plain}
\newtheorem{cor}[thm]{\protect\corollaryname}
\theoremstyle{remark}
\newtheorem{rem}[thm]{\protect\remarkname}
\theoremstyle{plain}
\newtheorem{lem}[thm]{\protect\lemmaname}
\theoremstyle{plain}
\newtheorem{prop}[thm]{\protect\propositionname}

\makeatother

\usepackage{babel}
\addto\captionsamerican{\renewcommand{\corollaryname}{Corollary}}
\addto\captionsamerican{\renewcommand{\lemmaname}{Lemma}}
\addto\captionsamerican{\renewcommand{\propositionname}{Proposition}}
\addto\captionsamerican{\renewcommand{\remarkname}{Remark}}
\addto\captionsamerican{\renewcommand{\theoremname}{Theorem}}
\addto\captionsenglish{\renewcommand{\corollaryname}{Corollary}}
\addto\captionsenglish{\renewcommand{\lemmaname}{Lemma}}
\addto\captionsenglish{\renewcommand{\propositionname}{Proposition}}
\addto\captionsenglish{\renewcommand{\remarkname}{Remark}}
\addto\captionsenglish{\renewcommand{\theoremname}{Theorem}}
\providecommand{\corollaryname}{Corollary}
\providecommand{\lemmaname}{Lemma}
\providecommand{\propositionname}{Proposition}
\providecommand{\remarkname}{Remark}
\providecommand{\theoremname}{Theorem}

\begin{document}
\title{The Congruence Subgroup Problem \\
for finitely generated Nilpotent Groups}
\author{David El-Chai Ben-Ezra, Alexander Lubotzky}
\maketitle
\begin{abstract}
The congruence subgroup problem for a finitely generated group $\Gamma$
and $G\leq Aut(\Gamma)$ asks whether the map $\hat{G}\to Aut(\hat{\Gamma})$
is injective, or more generally, what is its kernel $C\left(G,\Gamma\right)$?
Here $\hat{X}$ denotes the profinite completion of $X$. In the case
$G=Aut(\Gamma)$ we denote $C\left(\Gamma\right)=C\left(Aut(\Gamma),\Gamma\right)$.

Let $\Gamma$ be a finitely generated group, $\bar{\Gamma}=\Gamma/[\Gamma,\Gamma]$,
and $\Gamma^{*}=\bar{\Gamma}/tor(\bar{\Gamma})\cong\mathbb{Z}^{(d)}$.
Denote
\[
Aut^{*}(\Gamma)=\textrm{Im}(Aut(\Gamma)\to Aut(\Gamma^{*}))\leq GL_{d}(\mathbb{Z}).
\]
In this paper we show that when $\Gamma$ is nilpotent, there is a
canonical isomorphism $C\left(\Gamma\right)\simeq C(Aut^{*}(\Gamma),\Gamma^{*})$.
In other words, $C\left(\Gamma\right)$ is completely determined by
the solution to the classical congruence subgroup problem for the
arithmetic group $Aut^{*}(\Gamma)$.

In particular, in the case where $\Gamma=\Psi_{n,c}$ is a finitely
generated free nilpotent group of class $c$ on $n$ elements, we
get that $C(\Psi_{n,c})=C(\mathbb{Z}^{(n)})=\{e\}$ whenever $n\geq3$,
and $C(\Psi_{2,c})=C(\mathbb{Z}^{(2)})=\hat{F}_{\omega}$ = the free
profinite group on countable number of generators.\\
\end{abstract}
\textbf{}

\section{Introduction}

Let $G\leq GL_{n}\left(\mathbb{Z}\right)$. The classical congruence
subgroup problem (CSP) asks whether every finite index subgroup of
$G$ contains a principal congruence subgroup, i.e. a subgroup of
the form $G\left(m\right)=\ker\left(G\to GL_{n}\left(\mathbb{Z}/m\mathbb{Z}\right)\right)$
for some $0\neq m\in\mathbb{Z}$. Equivalently, it asks whether the
natural map $\hat{G}\to GL_{n}(\hat{\mathbb{Z}})$ is injective, where
$\hat{G}$ and $\hat{\mathbb{Z}}$ are the profinite completions of
the group $G$ and the ring $\mathbb{Z}$, respectively. More generally,
the CSP asks what is the kernel of this map. It is a classical $19^{\underline{th}}$
century result that for $G=GL_{n}\left(\mathbb{Z}\right)$ the answer
is negative when $n=2$. Moreover (but not so classical, cf. \cite{key-17},
\cite{key-4}), the kernel in this case is $\hat{F}_{\omega}$ - the
free profinite group on a countable number of generators. On the other
hand, it was proved in the sixties by Mennicke \cite{key-22} and
Bass-Lazard-Serre \cite{key-23} that for $n\geq3$ the map is injective,
and the kernel is therefore trivial. This breakthrough led to a rich
theory which studied the CSP for many other arithmetic groups. It
has been solved for many arithmetic groups, but not yet for all. See
\cite{key-15-1} and \cite{key-16} for surveys.

By the observation $GL_{n}\left(\mathbb{Z}\right)\cong Aut\left(\mathbb{Z}^{(n)}\right)$,
the CSP can be generalized as follows: Let $\Gamma$ be a group and
$G\leq Aut\left(\Gamma\right)$. For a finite index characteristic
subgroup $M\leq\Gamma$ denote
\[
G\left(M\right)=\ker\left(G\to Aut\left(\Gamma/M\right)\right).
\]
Such a $G\left(M\right)$ is called a ``principal congruence subgroup''
and a finite index subgroup of $G$ which contains $G\left(M\right)$
for some $M$ is called a ``congruence subgroup''. The CSP for the
pair $\left(G,\Gamma\right)$ asks whether every finite index subgroup
of $G$ is a congruence subgroup. 

One can see that the CSP is equivalent to the question: Is the congruence
map $\hat{G}=\underleftarrow{\lim}G/U\twoheadrightarrow\underleftarrow{\lim}G/G\left(M\right)$
injective? Here, $U$ ranges over all finite index normal subgroups
of $G$, and $M$ ranges over all finite index characteristic subgroups
of $\Gamma$. When $\Gamma$ is finitely generated, it has only finitely
many subgroups of a given index $m$, and thus, the charateristic
subgroups $M_{m}=\cap\left\{ \Delta\leq\Gamma\,|\,\left[\Gamma:\Delta\right]=m\right\} $
are of finite index in $\Gamma$. Hence, one can write $\Gamma=\underleftarrow{\lim}_{m\in\mathbb{N}}\Gamma/M_{m}$
and have\footnote{When we write $Aut(\hat{\Gamma})$ we basically mean to consider the
group of continuous automorphisms of $\hat{\Gamma}$. However, by
the celebrated theorem of Nikolov and Segal which asserts that every
finite index subgroup of a finitely generated profinite group is open
\cite{key-17-1}, whenever $\Gamma$ is finitely generated, the group
of continuous automorphisms of $\hat{\Gamma}$ is equal to the group
of automorphisms of $\hat{\Gamma}$. }
\begin{eqnarray*}
\underleftarrow{\lim}G/G\left(M\right) & = & \underleftarrow{\lim}_{m\in\mathbb{N}}G/G\left(M_{m}\right)\leq\underleftarrow{\lim}_{m\in\mathbb{N}}Aut(\Gamma/M_{m})\\
 & = & Aut(\underleftarrow{\lim}_{m\in\mathbb{N}}(\Gamma/M_{m}))=Aut(\hat{\Gamma}).
\end{eqnarray*}
Therefore, when $\Gamma$ is finitely generated, the CSP is equivalent
to the question: Is the congruence map $\hat{G}\to Aut(\hat{\Gamma})$
injective? More generally, the CSP asks what is the kernel $C\left(G,\Gamma\right)$
of this map. For $G=Aut\left(\Gamma\right)$ we will also use the
simpler notation $C\left(\Gamma\right)=C\left(Aut\left(\Gamma\right),\Gamma\right)$. 

The classical CSP results mentioned above can therefore be reformulated
as $C(\mathbb{Z}^{(2)})=\hat{F}_{\omega}$ while $C(\mathbb{Z}^{(n)})=\left\{ e\right\} $
for $n\geq3$. Recently, it was proved that when $\Gamma=\Phi_{n}$
is the free metabelian group on $n$ generators, we have: $C\left(\Phi_{2}\right)=\hat{F}_{\omega}$,
$C\left(\Phi_{3}\right)\supseteq\hat{F}_{\omega}$, and for every
$n\geq4$, $C\left(\Phi_{n}\right)$ is abelian (see \cite{key-6,key-14,key-6-2,key-7}).
I.e. while in the free abelain case there is a dichotomy between $n=2$
and $n\geq3$, in the free metabelian case we have dichotomy between
$n=2,3$ and $n\geq4$. 

The goal of this paper is to show that contrary to the above metabelian
cases, when $\Gamma$ is a finitely generated nilpotent group, the
CSP for $\Gamma$ is completely determined by the CSP for abelian
groups. Let us put things more precise: Let $\Gamma$ be a finitely
generated group, $\bar{\Gamma}=\Gamma/[\Gamma,\Gamma]$ and $\Gamma^{*}=\bar{\Gamma}/tor(\bar{\Gamma})$,
so $\Gamma^{*}\cong\mathbb{Z}^{(d)}$ for some $d$. Denote 
\[
Aut^{*}\left(\Gamma\right)=\textrm{Im}(Aut\left(\Gamma\right)\to Aut(\Gamma^{*}))\leq GL_{d}(\mathbb{Z}).
\]
When $\Gamma$ is nilpotent, the group $Aut^{*}(\Gamma)$ is known
to be an arithmetic subgroup of $GL_{d}(\mathbb{Z})$, and every arithmetic
subgroup of $GL_{d}(\mathbb{Z})$ is obtained like that for some nilpotent
group $\Gamma$ (\cite{key-1,key-31,key-32}).

The canonical map $Aut(\Gamma)\to Aut^{*}\left(\Gamma\right)$ induces
a map 
\[
C(\Gamma)\to C(Aut^{*}\left(\Gamma\right),\Gamma^{*}).
\]
Here is the main theorem of the paper:
\begin{thm}
\label{thm:Thm}Let $\Gamma$ be a finitely generated nilpotent group.
Then, the canonical map $C(\Gamma)\to C(Aut^{*}\left(\Gamma\right),\Gamma^{*})$
is an isomorphism.
\end{thm}

So, the CSP for nilpotent groups is completely reduced to the classic
CSP. In particular, in the free cases, we have:
\begin{cor}
\label{cor:free}Let $\Gamma=\Psi_{n,c}$ be the free nilpotent group
of class $c$ on $n$ elements. Then 
\[
C(\Gamma)\cong C(Aut^{*}\left(\Gamma\right),\Gamma^{*})=C(GL_{n}(\mathbb{Z}),\mathbb{Z}^{(n)})=C(\mathbb{Z}^{(n)}).
\]
In particular: 
\begin{itemize}
\item For $n=2$ one has $C(\Psi_{2,c})\cong C(\mathbb{Z}^{(2)})\cong\hat{F}_{\omega}$.
\item For $n\geq3$ one has $C(\Psi_{n,c})\cong C(\mathbb{Z}^{(n)})\cong\left\{ e\right\} $.
\end{itemize}
\end{cor}

\begin{rem}
As mentioned above, every arithmetic subgroup $D$ of $GL_{d}(\mathbb{Z})$
can appear as $Aut^{*}(\Gamma)$ for a suitable nilpotent $\Gamma$.
The possible congruence kernels for such arithmetic groups are not
fully known as the classical CSP is not yet fully solved. But these
include, besides the trivial groups and $\hat{F}_{\omega}$ mentioned
above, also finite cyclic groups (when $D$ is the restriction of
scalars from suitable number fields) as well as infinite abelian groups
of finite exponent (if $D$ is an arithmetic group of a non simply
connected group).
\end{rem}

Here is the main line of the proof. For a finitely generated group
$\Gamma$ consider the commutative exact diagram
\[
\begin{array}{ccccccccc}
1 & \to & IA^{*}\left(\Gamma\right) & \to & Aut\left(\Gamma\right) & \to & Aut^{*}\left(\Gamma\right) & \to & 1\\
 &  & \downarrow &  & \downarrow &  & \downarrow\\
1 & \to & IA^{*}(\hat{\Gamma}) & \to & Aut(\hat{\Gamma}) & \to & Aut^{*}(\hat{\Gamma}) & \to & 1
\end{array}
\]
when we define $IA^{*}(\Gamma)=\ker(Aut(\Gamma)\to Aut(\Gamma^{*}))$
and $Aut^{*}(\hat{\Gamma}),IA^{*}(\hat{\Gamma})$ defined to be the
image and the kernel of the natural map $Aut(\hat{\Gamma})\to Aut(\widehat{\Gamma^{*}})=GL_{d}(\mathbb{\hat{Z}})$,
respectively. This diagram gives rise to the commutative exact diagram
(see Lemma 2.1 in \cite{key-30})
\begin{equation}
\begin{array}{ccccccccc}
 &  & \widehat{IA^{*}\left(\Gamma\right)} & \to & \widehat{Aut\left(\Gamma\right)} & \to & \widehat{Aut^{*}\left(\Gamma\right)} & \to & 1\\
 &  & \downarrow &  & \downarrow &  & \downarrow\\
1 & \to & IA^{*}(\hat{\Gamma}) & \to & Aut(\hat{\Gamma}) & \to & Aut^{*}(\hat{\Gamma}) & \to & 1
\end{array}\label{eq:diagram}
\end{equation}
Notice that $C(Aut^{*}\left(\Gamma\right),\Gamma^{*})=\ker(\widehat{Aut^{*}\left(\Gamma\right)}\to Aut^{*}(\hat{\Gamma}))$.
We prove the following theorem:
\begin{thm}
\label{thm:The therorem}Let $\Gamma$ be a finitely generated nilpotent
group. Then:
\begin{enumerate}
\item For any $G\leq IA^{*}\left(\Gamma\right)$, the natural map $G\to IA^{*}(\hat{\Gamma})$
is an embedding. In other words 
\[
C(G,\Gamma)=\{e\}
\]
so we have an affirmative solution to the CSP for any $G\leq IA^{*}\left(\Gamma\right)$. 
\item The group $IA^{*}\left(\Gamma\right)$ is dense in $IA^{*}(\hat{\Gamma})$. 
\end{enumerate}
\end{thm}

Notice that from the first part of Theorem \ref{thm:The therorem}
we obtain that in particular $C(IA^{*}\left(\Gamma\right),\Gamma)=\{e\}$
for any finitely generated nilpotent group $\Gamma$. This is not
true in general (cf. \cite{key-6,key-14,key-6-2,key-7} for free metabelian
groups). In some sense, the second part of Theorem \ref{thm:The therorem}
means that the map $IA^{*}\left(\Gamma\right)\to IA^{*}(\hat{\Gamma})$
satisfies a ``strong approximation'' propery. This is not true in
general either (compare \cite{key-9} for free groups). From the two
parts of Theorem \ref{thm:The therorem} we obtain the following corollary,
which also implies Theorem \ref{thm:Thm}:
\begin{cor}
\label{cor:isomorphism}Let $\Gamma$ be a finitely generated nilpotent
group. Then
\[
\widehat{IA^{*}\left(\Gamma\right)}\cong IA^{*}(\hat{\Gamma}).
\]
\end{cor}

Corollary \ref{cor:isomorphism}, together with chasing diagram (\ref{eq:diagram}),
imply Theorem \ref{thm:Thm}. Corollary \ref{cor:isomorphism} is
a form of combination of congruence subgroup property as well as strong
approximation for the group $IA^{*}\left(\Gamma\right)$. Indeed,
its proof boils down to these results for a suitable $\mathbb{Q}$-unipotent
group. But the reduction is slightly delicate: in $\mathsection$\ref{sec:torsion},
it is shown that the proof of Corollary \ref{cor:isomorphism} (or
Theorem \ref{thm:The therorem}) can be reduced to the case when $\Gamma$
is torsion free. In $\mathsection$\ref{sec:torsion free}, we treat
the torsion free case, by reducing it first from $\Gamma$ to $\Delta$,
when $\Delta$ is the \textquotedbl lattice hull\textquotedbl{} of
$\Gamma$. This $\Delta$ contains $\Gamma$ as a finite index subgroup
and it is contained in its Mal'cev completion $R$. It enjoys the
property that $\log(\Delta)$ is a $\mathbb{Z}$-lattice of the Lie
algebra $L$ of $R$. This fact enables us to give $IA^{*}\left(\Delta\right)$
the structure of the $\mathbb{Z}$-points of a suitable unipotent
group for which the $\hat{\mathbb{Z}}$-points are exactly $IA^{*}(\hat{\Delta})$.
Hence, the classical CSP and strong approximation for this unipotent
group imply Corollary \ref{cor:isomorphism}.

In $\mathsection$\ref{sec:free} we sketch another proof to Corollary
\ref{cor:isomorphism}, which is more direct, in the case where $\Gamma=\Psi_{n,c}$
is a finitely generated free nilpotent group.

Acknowledgments: During the period of the research, the first author
was supported by the Rudin foundation and, not concurrently, by NSF
research training grant (RTG) \# 1502651. The second author is indebted
for support from the NSF (Grant No. DMS-1700165) and the European
Research Council (ERC) under the European Unions Horizon 2020 research
and innovation program (Grant No. 692854).

\section{\label{sec:torsion free}The case of Torsion Free Nilpotent Groups}

In this section we are going to prove Theorem \ref{thm:The therorem}
in the case where $\Gamma$ is torsion free. So let $\Gamma$ be a
finitely generated torsion free nilpotent group. For our convenience,
we will follow the approach presented in \cite{key-32-1}, and consider
$\Gamma$ as a subgroup of $Tr_{1}(n,\mathbb{Z})$, the group of $n\times n$
upper triangular matrices over $\mathbb{Z}$ with $1$-s on the diagonal,
for some $n$ (see Chapter 5 therein). Recall the one-to-one correspondence
given by the maps
\begin{align*}
\log: & Tr_{1}(n,\mathbb{Q})\to Tr_{0}(n,\mathbb{Q})\\
\exp: & Tr_{0}(n,\mathbb{Q})\to Tr_{1}(n,\mathbb{Q})
\end{align*}
where $Tr_{0}(n,\mathbb{Q})$ is the Lie algebra of $n\times n$ upper
triangular matrices with $0$-s on the diagonal. Let $L$ be the Lie
subalgebrs of $Tr_{0}(n,\mathbb{Q})$ spanned by $\log(\Gamma)$.
The following is well known and can be found in \cite{key-32-1} as
well (Chapter 6):
\begin{thm}
\label{thm:malcev}There exists a unique (up to isomorphism) group
$R$, called \uline{the radicable hull of \mbox{$\Gamma$}}, or
\uline{Mal'cev completion of \mbox{$\Gamma$}} with the following
properties:
\begin{itemize}
\item $\Gamma$ is a subgroup of $R$.
\item For every $a\in R$ and $m\in\mathbb{N}$ there exists $b\in R$ such
that $b^{m}=a$.
\item For every $a\in R$ there exists $m\in\mathbb{N}$ such that $a^{m}\in\Gamma$.
\item The group $R$ can be identified with $\exp(L)\leq Tr_{1}(n,\mathbb{Q})$.
\end{itemize}
\end{thm}

The connection between the group operation of $R$ and the Lie algebra
opration of $L$ is given through the Baker-Campbell-Hausdorf (BCH)
formula. One can use it in order to prove the following lemma (\cite{key-31},
Lemma 2.1):
\begin{lem}
\label{lem:prime}Under the correspondence between the underlying
sets of $R$ and $L$ one has $R'=L'$. This equality gives a natural
group isomorphism between $R/R'$ and the additive group $L/L'$. 
\end{lem}

One can use the BCH formula in order to prove that $L'$ is the $\mathbb{Q}$-span
of $\log(\Gamma')$, and hence $R'$ can be identified with the radicable
hull of $\Gamma'$. Denote $\delta(\Gamma)=\ker(\Gamma\to\Gamma^{*}\simeq\mathbb{Z}^{d})$.
Then, as any element of $\delta(\Gamma)$ has some power in $\Gamma'$,
we have $\log(\delta(\Gamma))\subseteq L'$, and hence $L'$ is also
the $\mathbb{Q}$-span of $\log(\delta(\Gamma))$, and $R'$ is also
the radicable hull of $\delta(\Gamma)$. One gets from this that (\cite{key-31},
Lemma 2.2):
\begin{lem}
\label{lem:delta}We have $\Gamma\cap R'=\delta(\Gamma)$ and $\dim_{\mathbb{Q}}(R/R')=\textrm{rank}_{\mathbb{Z}}(\Gamma/\delta(\Gamma))=d$.
\end{lem}

The following can be found in \cite{key-32-1}, Chapter 6:
\begin{prop}
\label{prop:lattice}There exists a unique minimal intermediate subgroup
$\Gamma\leq\Delta\leq R$, called the \uline{lattice hull} of $\Gamma$,
that its image in $L$, namely $\log(\Delta)$, is a lattice. I.e.
$\log(\Delta)$ is a free $\mathbb{Z}$-module that spans $L$ over
$\mathbb{Q}$. One has $[\Delta:\Gamma]<\infty$.
\end{prop}

\begin{rem}
Notice that $R$ is also the radicable hull of $\Delta$. 
\end{rem}

Given a set $X\subseteq L$ denote $N_{Aut(L)}(X)=\{g\in Aut(L)\,|\,g(X)=X\}$.
The following can also be found in \cite{key-32-1}, Chapter 6:
\begin{thm}
\label{thm:aut-corr}The correspondence between $R$ and $L$ induces
an isomorphism
\[
Aut(R)\simeq Aut(L)\leq GL_{k}(\mathbb{Q})
\]
where $k=\dim(L)$. Under this isomorphism, one can identify
\begin{align*}
Aut(\Gamma) & \cong N_{Aut(L)}(\log(\Gamma))\leq Aut(L)\\
Aut(\Delta) & \cong N_{Aut(L)}(\log(\Delta))\leq Aut(L).
\end{align*}
From Proposition \ref{prop:lattice}, it follows that $Aut(\Gamma)\leq Aut(\Delta)\leq Aut(R)\simeq Aut(L)$.
I.e. any automorphism of $\Gamma$ can be uniquely extended to an
automorphism of $\Delta$, and any of the latter can be uniquely extended
to an automorphism of $R$.
\end{thm}

Theorem \ref{thm:aut-corr} and Lemma \ref{lem:delta} imply:
\begin{prop}
\label{prop:IA*}Under the above notation we can identify 
\begin{align*}
IA^{*}(\Gamma) & =Aut(\Gamma)\cap\ker(Aut(R)\to Aut(R/R'))\\
 & \cong N_{Aut(L)}(\log(\Gamma))\cap\ker(Aut(L)\to Aut(L/L'))
\end{align*}
\begin{align*}
IA^{*}(\Delta) & =Aut(\Delta)\cap\ker(Aut(R)\to Aut(R/R'))\\
 & \cong N_{Aut(L)}(\log(\Delta))\cap\ker(Aut(L)\to Aut(L/L')).
\end{align*}
\end{prop}

An immediate corollary of Proposition \ref{prop:IA*} and Theorem
\ref{thm:aut-corr} is that $IA^{*}(\Gamma)$ is naturally embedded
in $IA^{*}(\Delta)$. We would like now to show that the same property
is valid also for the profinte completions of $\Gamma$ and $\Delta$.
Namely:
\begin{prop}
\label{prop:profinite completion}Any automorphism of $\hat{\Gamma}$
can be uniquely extended to an automorphism of $\hat{\Delta}$. In
particular, $IA^{*}(\hat{\Gamma})$ is naturally embedded as a finite
index subgroup of $IA^{*}(\hat{\Delta})$.
\end{prop}

In order to prove Proposition \ref{prop:profinite completion}, we
are going to show that one can describe the relation between $IA^{*}(\hat{\Gamma})$
and $IA^{*}(\hat{\Delta})$ in a very similar way to the description
of the relation between $IA^{*}(\Gamma)$ and $IA^{*}(\Delta)$ above.
Before we do that, let us present an immediate consequence of Proposition
\ref{prop:profinite completion}:
\begin{prop}
\label{prop:moving}Let $\Gamma$ be a finitely generated torsion
free nilpotent group, and let $\Delta$ be the lattice hull of $\Gamma$.
Let $G\leq IA^{*}(\Gamma)\leq IA^{*}(\Delta)$. Then:
\begin{enumerate}
\item If $\hat{G}\to IA^{*}(\hat{\Delta})$ is injective, then $\hat{G}\to IA^{*}(\hat{\Gamma})$
is injective.
\item If $IA^{*}(\Delta)$ is dense in $IA^{*}(\hat{\Delta})$, then $IA^{*}(\Gamma)$
is dense in $IA^{*}(\hat{\Gamma})$. 
\end{enumerate}
\end{prop}

\begin{proof}
The first statment is an immediate corollary of Proposition \ref{prop:profinite completion}.
The second part also follows immediately from Proposition \ref{prop:profinite completion}
since $IA^{*}(\Delta)\cap IA^{*}(\hat{\Gamma})=IA^{*}(\Gamma)$.
\end{proof}
Proposition \ref{prop:moving} shows us that in order to prove Theorem
\ref{thm:The therorem} for finitely generated torsion free nilpotent
group $\Gamma$, it is enough to show it for its lattice hull $\Delta$.
We turn now to describe the relation between $Aut(\hat{\Gamma})$
and $Aut(\hat{\Delta})$. The description is going to give more than
just a proof to Proposition \ref{prop:profinite completion}, and
we are going to use it also toward the rest of the section. 

Let $\Gamma_{p}$ be the pro-$p$ completion of $\Gamma$. As $\Gamma$
is nilpotent, we have $\hat{\Gamma}=\prod_{p}\Gamma_{p}$. In addition,
as $\Gamma$ is finitely generated and unipotent, it is arithmetic
(\cite{key-32-1}, Chapter 6). Hence, by the affirmative solution
to the congruence subgroup problem for arithmetic soluble groups (see
\cite{key-33,key-34,key-35}), we can view $\hat{\Gamma}$ as the
closure of $\Gamma$ under the map
\[
\Gamma\hookrightarrow Tr_{1}(n,\mathbb{Z})\to Tr_{1}(n,\hat{\mathbb{Z}})=\prod_{p}Tr_{1}(n,\mathbb{Z}_{p}).
\]
As $Tr_{1}(n,\mathbb{Z}_{p})$ is a pro-$p$ group, and $\hat{\Gamma}=\prod_{p}\Gamma_{p}$,
it follows that we can identify $\Gamma_{p}$ with the closure of
$\Gamma$ under the map
\[
\Gamma\hookrightarrow Tr_{1}(n,\mathbb{Z})\to\prod_{p}Tr_{1}(n,\mathbb{Z}_{p})\to Tr_{1}(n,\mathbb{Z}_{p}).
\]

Extending $\log:Tr_{1}(n,\mathbb{Q}_{p})\to Tr_{0}(n,\mathbb{Q}_{p})$
and $\exp:Tr_{0}(n,\mathbb{Q}_{p})\to Tr_{1}(n,\mathbb{Q}_{p})$,
$\log$ and $\exp$ are continuous with relation to the topology induced
by $\mathbb{Q}_{p}$. We define $L_{p}$ to be the $\mathbb{Q}$-span
of $\log(\Gamma_{p})$ and $R_{p}=\exp(L_{p})$.
\begin{lem}
\label{lem:Q-p algebra}The set $L_{p}$ is a $\mathbb{Q}_{p}$-Lie
algebra.
\end{lem}

\begin{proof}
The BCH clearly gives $L_{p}$ a structure of a $\mathbb{Q}$-Lie
algebra, just like it gives $L$. We just need to explain why $L_{p}$
is closed under multiplication of scalars from $\mathbb{Q}_{p}$.
By definition, it is enough to show that it is closed under multiplication
of scalars from $\mathbb{Z}_{p}$. So let $g\in\Gamma_{p}$ and let
$m=\underset{i\to\infty}{\lim}m_{i}\in\mathbb{Z}_{p}$ for some $m_{i}\in\mathbb{Z}$.
Let $\rho_{g}$ be the natural homomorphism $\rho_{g}:\mathbb{Z}_{p}\to\Gamma_{p}$
defined by sending the generator of $\mathbb{Z}_{p}$ to $g$. Then,
as $\log$ is continuous we have
\begin{align*}
\log(\rho_{g}(m)) & =\log(\rho_{g}(\underset{i\to\infty}{\lim}m_{i}))=\underset{i\to\infty}{\lim}\log(\rho_{g}(m_{i}))=\underset{i\to\infty}{\lim}\log(g^{m_{i}})\\
 & =\underset{i\to\infty}{\lim}(m_{i}\cdot\log(g))=(\underset{i\to\infty}{\lim}m_{i})\cdot\log(g)=m\cdot\log(g).
\end{align*}
It follows that the set $\log(\Gamma_{p})$ is closed under multiplication
by elements from $\mathbb{Z}_{p}$, and so is $L_{p}$.
\end{proof}
The proof of the following is similar to the proof of the corresponding
properties of the Mal'cev completion in Theorem \ref{thm:malcev}:
\begin{lem}
\label{lem:R_p'}The set $R_{p}$ is a group containing $\Gamma,R$
and $\Gamma_{p}$. Moreover, $R_{p}$ is a Mal'cev completion of $\Gamma_{p}$
in the sense that it satisfies the following properties:
\begin{itemize}
\item For every $a\in R_{p}$ and $m\in\mathbb{N}$ there exists $b\in R_{p}$
such that $b^{m}=a$. 
\item For every $a\in R_{p}$ there exists $m\in\mathbb{N}$ such that $a^{m}\in\Gamma_{p}$.
\end{itemize}
\end{lem}

Also, similarly to Lemma \ref{lem:prime}, one can use the BCH formula
in order to show:
\begin{lem}
\label{lem:R-p-L-p}Under the correspondence between $R_{p}$ and
$L_{p}$ one has $R_{p}'=L_{p}'$. This equality gives a natural group
homomorphism between $R_{p}/R_{p}'$ and the additive group $L_{p}/L_{p}'$. 
\end{lem}

One can use the BCH formula in order to prove that $L_{p}'$ is the
$\mathbb{Q}$-span of $\log(\Gamma_{p}')$, and hence $R_{p}'$ is
a Mal'cev completion of $\Gamma_{p}'$ in the sense of Lemma \ref{lem:R_p'}.
Denoting $\delta(\Gamma_{p})=\ker(\Gamma_{p}\to(\Gamma^{*})_{p}\simeq\mathbb{Z}_{p}^{d})$
where $d=\textrm{rank}_{\mathbb{Z}}(\Gamma^{*})$, we have also a
similar property as in Lemma \ref{lem:delta}:
\begin{lem}
\label{lem:delta-p}One has $\Gamma_{p}\cap R_{p}'=\delta(\Gamma_{p})$
and 
\[
\dim_{\mathbb{Q}_{p}}(R_{p}/R_{p}')=\textrm{rank}_{\mathbb{Z}_{p}}(\Gamma_{p}/\delta(\Gamma_{p}))=d.
\]
\end{lem}

\begin{proof}
We first prove that $\delta(\Gamma_{p})=\ker(\Gamma_{p}\to\bar{\Gamma}_{p}/tor(\bar{\Gamma}_{p}))$,
where $\bar{\Gamma}_{p}=\Gamma_{p}/\Gamma_{p}'$. As $(\Gamma^{*})_{p}$
is a torsion free abelian quotient of $\Gamma_{p}$, it follows that
$\delta(\Gamma_{p})\supseteq\ker(\Gamma_{p}\to\bar{\Gamma}_{p}/tor(\bar{\Gamma}_{p}))$.
On the other hand, the map $\Gamma\to\Gamma_{p}$ induces a map $\Gamma^{*}\to\text{\ensuremath{\bar{\Gamma}_{p}}/tor(\ensuremath{\bar{\Gamma}_{p}})}$.
Now, as $\Gamma_{p}$ is a finitely generated pro-$p$ group, $\Gamma_{p}'$
is closed in $\Gamma_{p}$ (see Proposition 1.19 in \cite{key-36}).
Hence, $\Gamma_{p}\to\bar{\Gamma}_{p}$ is a continuous homomorphism,
and hence, $\Gamma_{p}\to\bar{\Gamma}_{p}/tor(\bar{\Gamma}_{p})$
is continuous as well. Thus, we can say that the image of $\Gamma^{*}$
under the map $\Gamma^{*}\to\text{\ensuremath{\bar{\Gamma}_{p}}/tor(\ensuremath{\bar{\Gamma}_{p}})}$,
is dense in $\bar{\Gamma}_{p}/tor(\bar{\Gamma}_{p})$. Hence, we have
a surjective continuous homomorphism $(\Gamma^{*})_{p}\to\bar{\Gamma}_{p}/tor(\bar{\Gamma}_{p})$.
It follows that $\delta(\Gamma_{p})\subseteq\ker(\Gamma_{p}\to\bar{\Gamma}_{p}/tor(\bar{\Gamma}_{p}))$,
so $\delta(\Gamma_{p})=\ker(\Gamma_{p}\to\bar{\Gamma}_{p}/tor(\bar{\Gamma}_{p}))$
as required.

Now, as $R_{p}$ contains $\Gamma_{p}$, we have $\Gamma_{p}/(\Gamma_{p}\cap R_{p}')\leq R_{p}/R_{p}'$.
It follows that $\Gamma_{p}/(\Gamma_{p}\cap R_{p}')$ is a torsion
free abelian quotient of $\Gamma_{p}$. Hence
\[
\delta(\Gamma_{p})=\ker(\Gamma_{p}\to\bar{\Gamma}_{p}/tor(\bar{\Gamma}_{p}))\subseteq\Gamma_{p}\cap R_{p}'.
\]
On the other hand, as every element of $R_{p}'$ has a power in $\Gamma_{p}'$,
it follows that we also have $\Gamma_{p}\cap R_{p}'\subseteq\ker(\Gamma_{p}\to\bar{\Gamma}_{p}/tor(\bar{\Gamma}_{p}))=\delta(\Gamma_{p})$,
as required.

The equality $\Gamma_{p}\cap R_{p}'=\delta(\Gamma_{p})$ implies that
$\Gamma_{p}/\delta(\Gamma_{p})\leq R_{p}/R_{p}'$. Thus $\Gamma_{p}/\delta(\Gamma_{p})$
is a free $\mathbb{Z}_{p}$-module that spanes the vector space $R_{p}/R_{p}'$
over $\mathbb{Q}_{p}$. Hence $\dim_{\mathbb{Q}_{p}}(R_{p}/R_{p}')=\textrm{rank}_{\mathbb{Z}_{p}}(\Gamma_{p}/\delta(\Gamma_{p}))=d$,
as required.
\end{proof}
The following corollary will be needed later:
\begin{cor}
\label{cor:dim}We have $\dim_{\mathbb{Q}}(L)=\dim_{\mathbb{Q}_{p}}(L_{p})$.
\end{cor}

\begin{proof}
We saw that 
\begin{align*}
\dim_{\mathbb{Q}_{p}}(L_{p}/L_{p}') & =\dim_{\mathbb{Q}_{p}}(R_{p}/R_{p}')=\textrm{rank}_{\mathbb{Z}_{p}}(\Gamma_{p}/\delta(\Gamma_{p}))=d\\
 & =\textrm{rank}_{\mathbb{Z}}(\Gamma/\delta(\Gamma))=\dim_{\mathbb{Q}}(R/R')=\dim_{\mathbb{Q}}(L/L').
\end{align*}
By the fact that the commutator subgroup of a finitely generated niplotent
group is finitely generated, and $L'$ is the $\mathbb{Q}$-span of
$\log(\Gamma')$ one has
\[
\textrm{rank}_{\mathbb{Z}_{p}}((\Gamma')_{p}/\delta((\Gamma')_{p}))=\textrm{rank}_{\mathbb{Z}}(\Gamma'/\delta(\Gamma'))=\dim_{\mathbb{Q}}(L'/L'').
\]
Using the CSP for arithmetic soluble groups, we can identify $(\Gamma')_{p}$
with the closure of $\Gamma'$ in $\Gamma_{p}\leq Tr_{1}(n,\mathbb{Z}_{p})$.
As explained in the proof of Lemma \ref{lem:delta-p}, the latter
can be identified with $(\Gamma_{p})'$, and hence $(\Gamma')_{p}=(\Gamma_{p})'$.
Hence, $L_{p}'$, which is the $\mathbb{Q}$-span of $\log((\Gamma_{p})')$,
is actually the $\mathbb{Q}$-span of $\log((\Gamma')_{p})$. It follows
that
\[
\dim_{\mathbb{Q}_{p}}(L_{p}'/L_{p}'')=\textrm{rank}_{\mathbb{Z}_{p}}((\Gamma')_{p}/\delta((\Gamma')_{p}))=\dim_{\mathbb{Q}}(L'/L'').
\]
Continuing like that, we obtain that $\dim_{\mathbb{Q}}(L^{(i)}/L^{(i+1)})=\dim_{\mathbb{Q}_{p}}(L_{p}^{(i)}/L_{p}^{(i+1)})$
for any $i$, where $L^{(i)},L_{p}^{(i)}$ are the $i$-th derivatives
of $L,L_{p}$ respectively. Therefore
\[
\dim_{\mathbb{Q}}(L)=\sum_{i}\dim_{\mathbb{Q}}(L^{(i)}/L^{(i+1)})=\sum_{i}\dim_{\mathbb{Q}_{p}}(L_{p}^{(i)}/L_{p}^{(i+1)})=\dim_{\mathbb{Q}_{p}}(L_{p})
\]
as required. 
\end{proof}
\begin{rem}
We presented a proof for Corollary \ref{cor:dim}, based on the previous
line of discussion, but the knowledgeable reader can also deduce it
by recalling that $\dim_{\mathbb{Q}}(L)$ is equal to $h(\Gamma)$
- the Hirsch length of $\Gamma$, and $h(\Gamma)$ is equal to $\dim(\Gamma_{p})$
- the dimension of the pro-$p$ completion of $\Gamma$, and the later
is equal to $\dim_{\mathbb{Q}_{p}}(L_{p})$ as $L_{p}$ is the Lie
algebra of $\Gamma_{p}$.
\end{rem}

Recall $\Delta$, the lattice hull of $\Gamma$. For our convenience,
without loss of generality, we can assume that $\Gamma\leq\Delta\leq Tr_{1}(n,\mathbb{Z})$
(see Lemma 2 in Chapter 6 of \cite{key-32-1}). Hence, $\Delta_{p}$
can be identified with the closure of $\Delta$ in $Tr_{1}(n,\mathbb{Z}_{p})$. 
\begin{lem}
\label{lem:log-p}One has $\log(\Delta_{p})=\mathbb{Z}_{p}\log(\Delta)\subseteq L_{p}$,
where $\mathbb{Z}_{p}\log(\Delta)$ is the $\mathbb{Z}_{p}$-lattice
of $L_{p}$ sppaned by $\log(\Delta)$. In particular, $\log(\Delta_{p})$
is a $\mathbb{Z}_{p}$-lattice in $L_{p}$ and $\Gamma_{p}\leq\Delta_{p}\leq R_{p}$.
\end{lem}

\begin{proof}
By a similar argument as in Lemma \ref{lem:Q-p algebra}, for any
element $m\in\mathbb{Z}_{p}$ and $a\in\log(\Delta)$ we have $m\cdot a\in\log(\Delta_{p})$.
Hence, we will get $\log(\Delta_{p})\supseteq\mathbb{Z}_{p}\log(\Delta)$
once we show that for any $g,h\in\Delta_{p}$ there exists $k\in\Delta_{p}$
such that $\log(g)+\log(h)=\log(k)$. By assumption, the group $\Delta$
is dense in $\Delta_{p}$ where the topology on $\Delta_{p}$ coincides
with the topology induced by $Tr_{1}(\mathbb{Q}_{p})$. Let $g_{i},h_{i},k_{i}\in\Delta$
be such that
\[
\underset{i\to\infty}{\lim}g_{i}=g\,\,\,\,\,\,\,\,\underset{i\to\infty}{\lim}h_{i}=h\,\,\,\,\,\,\,\,\log(k_{i})=\log(g_{i})+\log(h_{i}).
\]
We need to show that $\underset{i\to\infty}{\lim}k_{i}=k\in\Delta_{p}$
exists and $\log(g)+\log(h)=\log(k)$. As $\exp$ and $\log$ are
continuous, we have
\begin{align*}
k & =\underset{i\to\infty}{\lim}k_{i}=\underset{i\to\infty}{\lim}\exp(\log(g_{i})+\log(h_{i}))\\
 & =\exp(\log(\underset{i\to\infty}{\lim}g_{i})+\log(\underset{i\to\infty}{\lim}h_{i}))=\exp(\log(g)+\log(h)).
\end{align*}
As $\Delta_{p}$ is compact, we have $k\in\Delta_{p}$ where $\log(g)+\log(h)=\log(k)$.

For the opposite inclusion: $\Delta$ is dense in $\Delta_{p}$, so
$\log(\Delta)$ is dense in $\log(\Delta_{p})$. As $\mathbb{Z}_{p}\log(\Delta)$
is closed in $L_{p}$, it follows that $\log(\Delta_{p})\subseteq\mathbb{Z}_{p}\log(\Delta)$,
as required.
\end{proof}
\begin{lem}
\label{lem:minimality^}The group $\Delta_{p}$ is a unique minimal
subgroup $\Gamma_{p}\leq\Delta_{p}\leq R_{p}$ that its image in $L_{p}$
is a $\mathbb{Z}_{p}$-lattice.
\end{lem}

\begin{proof}
Suppose that $\Lambda_{p}$ satisfies $\Gamma_{p}\leq\Lambda_{p}\leq R_{p}$,
$\log(\Lambda_{p})$ is a $\mathbb{Z}_{p}$-lattice, and $\log(\Delta_{p})\nsubseteq\log(\Lambda_{p})$.
It follows that $\Gamma\leq\Lambda_{p}\cap\Delta\lvertneqq\Delta$
where $\log(\Lambda_{p}\cap\Delta)=\log(\Lambda_{p})\cap\log(\Delta)$
is a lattice of $L$ that contains $\log(\Gamma)$. This is a contradiction
to the minimality of $\Delta$.
\end{proof}
Now, using Lemma \ref{lem:R_p'} and the property $\Gamma_{p}\leq\Delta_{p}\leq R_{p}$
(Lemma \ref{lem:log-p}), and following the same steps for the proof
of \cite{key-32-1} to Theorem \ref{thm:aut-corr}, we have:
\begin{thm}
\label{thm:identification^}There are natural isomorphisms
\begin{align*}
Aut(R_{p}) & \cong Aut(L_{p})\\
Aut(\Gamma_{p}) & \cong N_{Aut(L_{p})}(\log(\Gamma_{p}))\leq Aut(L_{p})\\
Aut(\Delta_{p}) & \cong N_{Aut(L_{p})}(\log(\Delta_{p}))\leq Aut(L_{p}).
\end{align*}
\end{thm}

By Lemmas \ref{lem:R-p-L-p}, \ref{lem:delta-p}, and Theorem \ref{thm:identification^},
we obtain:
\begin{prop}
\label{prop:IA-identification^}Denote $IA^{*}(\Gamma_{p})=\ker(Aut(\Gamma_{p})\to Aut((\Gamma^{*})_{p})$.
Then
\begin{align*}
IA^{*}(\Gamma_{p}) & \cong N_{Aut(L_{p})}(\log(\Gamma_{p}))\cap\ker(Aut(L_{p})\to Aut(L_{p}/L'_{p})).
\end{align*}
Similarly, as $\Gamma_{p}\leq\Delta_{p}\leq R_{p}$ (Lemma \ref{lem:log-p}),
we also have
\begin{align*}
IA^{*}(\Delta_{p}) & \cong N_{Aut(L_{p})}(\log(\Delta_{p}))\cap\ker(Aut(L_{p})\to Aut(L_{p}/L'_{p})).
\end{align*}
\end{prop}

We can now deduce Proposition \ref{prop:profinite completion}:
\begin{proof}
(of Proposition \ref{prop:profinite completion}) From Lemma \ref{lem:minimality^}
and Proposition \ref{prop:IA-identification^} we get that for any
prime $p$ one has $IA^{*}(\Gamma_{p})\hookrightarrow IA^{*}(\Delta_{p})$.
Therefore
\begin{align*}
IA^{*}(\hat{\Gamma}) & =\prod_{p}IA^{*}(\Gamma_{p})\hookrightarrow\prod_{p}IA^{*}(\Delta_{p})=IA^{*}(\hat{\Delta}).
\end{align*}
Moreover, $IA^{*}(\hat{\Gamma})=\{\alpha\in IA^{*}(\hat{\Delta})\,|\,\alpha(\hat{\Gamma})=\hat{\Gamma}\}$,
hence it is open in $IA^{*}(\hat{\Delta})$. 
\end{proof}
As we mentioned previously, by Proposition \ref{prop:moving}, in
order to prove Theorem \ref{thm:The therorem} for $\Gamma$, it is
enough to prove it for $\Delta$. So from now on $\log(\Delta)$ is
a lattice that spans $L$, and hence $k=\dim_{\mathbb{Q}}(L)=\textrm{rank}_{\mathbb{Z}}(\log(\Delta))$.
Our objective now is to construct a basis for $\log(\Delta)$ that
with relation to it, we will be able to view $IA^{*}(\Delta)$ as
the $\mathbb{Z}$-points of a $\mathbb{Q}$-algebraic group whose
$\mathbb{Z}_{p}$-points are $IA^{*}(\Delta_{p})$. Once we show this,
we will see that Theorem \ref{thm:The therorem} (1) follows from
the classical CSP for the arithmetic group $IA^{*}(\Delta)$, and
Theorem \ref{thm:The therorem} (2) is the classical strong-approximation
theorem for this group.

Let $g_{1},...,g_{d}\in\Delta$ be such that $\Delta^{*}=\Delta/\delta(\Delta)\simeq\mathbb{Z}^{d}$
is generated by the images $\bar{g}_{1},...,\bar{g}_{d}$. Then, $\bar{g}_{1},...,\bar{g}_{d}$
also generate $(\Delta^{*})_{p}=\Delta_{p}/\delta(\Delta_{p})\simeq\mathbb{Z}_{p}^{d}$
as a pro-$p$ group. Hence, by Lemmas \ref{lem:delta} and \ref{lem:delta-p},
$l_{1}=\log(g_{1}),...,l_{d}=\log(g_{d})$ generate $R/R'\cong L/L'\cong\mathbb{Q}^{(d)}$
over $\mathbb{Q}$ and generate $R_{p}/R_{p}'\cong L_{p}/L_{p}'\cong\mathbb{Q}_{p}^{(d)}$
over $\mathbb{Q}_{p}$. It follows that $l_{1},...,l_{d}$ are linearly
independent over $\mathbb{Q}_{p}$ (and over $\mathbb{Q}$) and since
$L$ and $L_{p}$ are nilpotent, they generate $L$ as a Lie algebra
over $\mathbb{Q}$ and generate $L_{p}$ as a Lie algebra over $\mathbb{Q}_{p}$. 
\begin{lem}
\label{lem:B}Let $L=\gamma_{1}(L),L'=\gamma_{2}(L),...,0=\gamma_{c}(L)$
be the lower central series of $L$. The set $l_{1},...,l_{d}$ can
be completed to a basis
\[
B=\{l_{1},...,l_{d},l_{d+1},...,l_{k}\}
\]
of the lattice $\log(\Delta)$ such that:
\begin{enumerate}
\item $B$ contains bases for $\gamma_{j}(L)/\gamma_{j+1}(L)$ mod $\gamma_{j+1}(L)$
for every $j$.
\item $l_{j+1}$ lies in the same term as $l_{j}$ in the lower central
series of $L$, or deeper.
\end{enumerate}
\end{lem}

\begin{proof}
Now, $l_{1},...,l_{d}$ lie in $\gamma_{1}(L)-\gamma_{2}(L)$ and
provide a basis for $L/L'=\gamma_{1}(L)/\gamma_{2}(L)$ mod $\gamma_{2}(L)$.
Denote $i_{1}=d$. 

For every $j\geq2$ the group $(\log(\Delta)\cap\gamma_{j}(L))/(\log(\Delta)\cap\gamma_{j+1}(L))$
is a lattice in $\gamma_{j}(L))/\gamma_{j+1}(L)$ and hence
\[
\textrm{rank}_{\mathbb{Z}}((\log(\Delta)\cap\gamma_{j}(L))/(\log(\Delta)\cap\gamma_{j+1}(L)))=\dim_{\mathbb{Q}}(\gamma_{j}(L))/\gamma_{j+1}(L)).
\]
Let $l_{i_{j-1}+1},...,l_{i_{j}}\in\log(\Delta)\cap\gamma_{j}(L)$
be a basis to 
\[
(\log(\Delta)\cap\gamma_{j}(L))/(\log(\Delta)\cap\gamma_{j+1}(L))\,\,\,\,\textrm{mod}\,\,\,\,\log(\Delta)\cap\gamma_{j+1}(L)).
\]
Then $l_{i_{j-1}+1},...,l_{i_{j}}$ gives a basis for $\gamma_{j}(L))/\gamma_{j+1}(L)$
mod $\gamma_{j+1}(L)$ and they all lie in $\gamma_{j}(L)-\gamma_{j+1}(L)$.

This procedure gives us $l_{1},...,l_{i_{c-1}}$ that satisfy the
two conditions in the lemma. Also, by the construction, $l_{1},...,l_{i_{c-1}}$
span the abelain group $\log(\Delta)$ and provide a basis for $L$.
It follows that $i_{c-1}=k$ and that $l_{1},...,l_{k}$ is the desired
basis for $\log(\Delta)$. 
\end{proof}
Let $B$ be a basis for $\log(\Delta)$ as in the above lemma. Cleary,
$B$ is also a basis for $L$ as a vector space over $\mathbb{Q}$.
As $L_{p}$ is generated as a Lie algebra by $l_{1},...,l_{d}$ over
$\mathbb{Q}_{p}$, the set $B$ also spans $L_{p}$ over $\mathbb{Q}_{p}$.
As $\dim_{\mathbb{Q}}(L)=\dim_{\mathbb{Q}_{p}}(L_{p})$ (Proposition
\ref{cor:dim}), it follows that $B$ is also a basis for $L_{p}$
as a vector space over $\mathbb{Q}_{p}$. Hence, using the basis $B$
we can identify
\[
Aut(L)=GL_{k}(\mathbb{Q})\cap Aut(L_{p})\leq GL_{k}(\mathbb{Q}_{p}).
\]
In addition, as $\log(\Delta)$ is a lattice and $\log(\Delta_{p})=\mathbb{Z}_{p}\log(\Delta)$
(Lemma \ref{lem:log-p}), using the basis $B$ we can identify 
\begin{align*}
Aut(\Delta)=N_{Aut(L)}(\log(\Delta)) & =GL_{k}(\mathbb{Z})\cap Aut(L)\\
Aut(\Delta_{p})=N_{Aut(L_{p})}(\log(\Delta_{p})) & =GL_{k}(\mathbb{Z}_{p})\cap Aut(L_{p}).
\end{align*}
Moreover, we can identify $Aut(L)$ and $Aut(L_{p})$ with the groups
\begin{align*}
Aut(L) & =\{A\in GL_{k}(\mathbb{Q})\,|\,[Ax,Ay]-A[x,y]=0\,\,\,\,\forall x,y\in B\}\\
Aut(L_{p}) & =\{A\in GL_{k}(\mathbb{Q}_{p})\,|\,[Ax,Ay]-A[x,y]=0\,\,\,\,\forall x,y\in B\}.
\end{align*}

Now, notice that $\sigma\in\ker(Aut(L)\to Aut(L/L'))$ if and only
if for every $i=1,...,d$ we have $\sigma(l_{i})\in l_{i}+L'$. As
$l_{1},...,l_{d}$ generate $L$ over $\mathbb{Q}$, it follows that
for every $\sigma\in\ker(Aut(L)\to Aut(L/L'))$ and every $i$ (not
only for $i=1,...,d$) we have
\[
\sigma(l_{i})=l_{i}+n_{i}
\]
where $n_{i}$ lies in a strickly deeper term in the lower central
series of $L$, than of the one that $l_{i}$ lies in. Hence, by the
construction in Lemma \ref{lem:B}, with relation to $B$ we have
\[
\ker(Aut(L)\to Aut(L/L'))=Aut(L)\cap Tr_{1}(k,\mathbb{Q})\leq GL_{k}(\mathbb{Q}_{p}).
\]
It follows that $IA^{*}(\Delta)$ can be identified with the arithmetic
group of $\mathbb{Z}$-points of $Aut(L)\cap Tr_{1}(k,\mathbb{Q})$.
Similarly, we have
\[
\ker(Aut(L_{p})\to Aut(L_{p}/L_{p}'))=Aut(L_{p})\cap Tr_{1}(k,\mathbb{Q}_{p})\leq GL_{k}(\mathbb{Q}_{p}).
\]
and hence $IA^{*}(\Delta_{p})$ can be identifiead with the $\mathbb{Z}_{p}$-points
of $Aut(L_{p})\cap Tr_{1}(k,\mathbb{Q}_{p})$. By the description
above, $Aut(L)\cap Tr_{1}(k,\mathbb{Q})$ and $Aut(L_{p})\cap Tr_{1}(k,\mathbb{Q}_{p})$
are unipotent algebraic groups which are defined by the same equations,
over $\mathbb{Q}$ and $\mathbb{Q}_{p}$ respectively. 

Now, $IA^{*}(\Delta)$ is a unipotent subgroup of $GL_{k}(\mathbb{Z})$.
Hence, by the congruence subgroup property for unipotent arithmetic
groups we have
\[
\begin{array}{cccc}
\widehat{IA^{*}(\Delta)}=\prod_{p}(IA^{*}(\Delta))_{p} & \hookrightarrow & \prod_{p}Tr_{1}(k,\mathbb{Z}_{p}) & =Tr_{1}(k,\mathbb{\hat{Z}})\\
 & \searrow & \uparrow\\
 &  & \prod_{p}IA^{*}(\Delta_{p}) & =IA^{*}(\hat{\Delta})
\end{array}
\]
and hence $\widehat{IA^{*}(\Delta)}\hookrightarrow IA^{*}(\hat{\Delta})$
is injective. As $IA^{*}(\Delta)$ is a finitely generated nilpotent
group, for any $G\leq IA^{*}(\Delta)$ we have $\hat{G}\hookrightarrow\widehat{IA^{*}(\Delta)}$,
and hence $\hat{G}\hookrightarrow IA^{*}(\hat{\Delta})$. This gives
the first part of Theorem \ref{thm:The therorem} in the case of finitely
generated torsion free nilpotent groups. 

For the second statement of Theorem \ref{thm:The therorem}, notice
that by the above description for $IA^{*}(\Delta)$ and $IA^{*}(\Delta_{p})$,
we obtain that $IA^{*}(\Delta)$ is dense in $IA^{*}(\Delta_{p})$
by the strong approximation property for arithmetic unipotent groups
(see \cite{key-36-1}, Proposition 7.1, and the corollary afterward).
Therefore, as each of the groups $IA^{*}(\Delta_{p})$ is a pro-$p$
group, we obtain that $IA^{*}(\Delta)$ is dense in 
\[
\prod_{p}IA^{*}(\Delta_{p})=IA^{*}(\hat{\Delta}).
\]
This completes the proof of Theorem \ref{thm:The therorem} for finitely
generated torsion free nilpotent groups.

\section{\label{sec:torsion}The general case}

The aim of this section is to prove Theorem \ref{thm:The therorem}
given its validity for finitely generated torsion free nilpotent group.
Along the section, $\Gamma$ is a finitely generated nilpotent group,
and $\Delta=\Gamma/tor(\Gamma)$. As a finitely generated nilpotent
group, $\Gamma$ is residually finite. Hence, we can think on $\Gamma$
as a subgroup of $\hat{\Gamma}$. One has $tor(\hat{\Gamma})=tor(\Gamma)$
(see Lemma 2.3 in \cite{key-10}, Corollary 7.5 in \cite{key-34}).
In other words, $tor(\Gamma)$ is a normal subgroup of $\hat{\Gamma}$,
and
\[
\hat{\Delta}=\widehat{\Gamma/tor(\Gamma)}=\hat{\Gamma}/tor(\Gamma)=\hat{\Gamma}/tor(\hat{\Gamma})
\]
is torsion free. In addition, we get that $tor(\Gamma)$ is characteristic,
not only as a subgroup of $\Gamma$, but also as a subgroup of $\hat{\Gamma}$.
Therefore, we have a map $IA^{*}(\hat{\Gamma})\to IA^{*}(\hat{\Delta})$
and we obtain the commutative diagram
\[
\begin{array}{ccc}
IA^{*}(\Gamma) & \to & IA^{*}(\Delta)\\
\downarrow &  & \downarrow\\
IA^{*}(\hat{\Gamma}) & \to & IA^{*}(\hat{\Delta})
\end{array}
\]
Denote $\tilde{K}=\ker(IA^{*}(\Gamma)\to IA^{*}(\Delta))$. Let $x_{1},...,x_{n}$
be a generating set for $\Gamma$. Then, every $\alpha\in\tilde{K}$
can be described by
\[
x_{i}\mapsto x_{i}a_{i}
\]
for some $a_{i}\in tor(\Gamma)$. As $tor(\Gamma)$ is finite, $\tilde{K}$
is also finite. 

Now, let $G\leq IA^{*}(\Gamma)$, and denote its image in $IA^{*}(\Delta)$
by $H$. Then $K=\ker(G\to H)\leq\tilde{K}$ is finite, and hence
we obtain the following commutative exact diagram
\[
\begin{array}{ccccccccc}
1 & \to & K & \to & G & \to & H & \to & 1\\
 &  &  & \searrow & \downarrow &  & \downarrow\\
 &  &  &  & IA^{*}(\hat{\Gamma}) & \to & IA^{*}(\hat{\Delta})
\end{array}
\]
Notice that as $G\leq IA^{*}(\Gamma)\hookrightarrow IA^{*}(\hat{\Gamma})$,
the map $K\to IA^{*}(\hat{\Gamma})$ is injective. Notice also that
as $K$ is finite, we have $K=\hat{K}$. Hence, moving to the profinite
completion of the upper row, we get the commutative exact diagram
\[
\begin{array}{ccccccc}
K=\hat{K} & \to & \hat{G} & \to & \hat{H} & \to & 1\\
 & \searrow & \downarrow &  & \downarrow\\
 &  & IA^{*}(\hat{\Gamma}) & \to & IA^{*}(\hat{\Delta})
\end{array}
\]
It follows that $\hat{K}\to IA^{*}(\hat{\Gamma})$ is injective and
by $\mathsection$\ref{sec:torsion free}, also $\hat{H}\to IA^{*}(\hat{\Delta})$
is injective. Hence, by diagram chasing, $\hat{G}\to IA^{*}(\hat{\Gamma})$
is also injective. This proves the first assertion of Theorem \ref{thm:The therorem}
in the general case. We move now to prove the second assertion.

As $\Gamma$ is residually finite, and $tor(\Gamma)$ is finite, there
exists $t\in\mathbb{N}$ such that $tor(\Gamma)\cap\Gamma^{t}=\{e\}$
where $\Gamma^{t}$ is the normal subgroup
\[
\Gamma^{t}=\left\langle g^{t}\,|\,g\in G\right\rangle \vartriangleleft\Gamma.
\]
Fix this $t$. It follows that whenever $t|m$, we have $tor(\Gamma)\cap\Gamma^{m}=\{e\}$,
and therefore $tor(\Gamma)$ is naturally embedded in $\Gamma/\Gamma^{m}$
for such $m$. As $\Gamma$ is nilpotent and finitely generated, $\Gamma/\Gamma^{m}$
and $\Delta/\Delta^{m}$ are finite for any $m$ (see Corollary 3.3
in \cite{key-34}), and hence one has
\begin{eqnarray*}
\hat{\Gamma} & = & \underset{m\in\mathbb{N}}{\underleftarrow{\lim}}\Gamma/\Gamma^{m}=\underset{t|m}{\underleftarrow{\lim}}\Gamma/\Gamma^{m}\\
\hat{\Delta} & = & \underset{m\in\mathbb{N}}{\underleftarrow{\lim}}\Delta/\Delta^{m}=\underset{t|m}{\underleftarrow{\lim}}\Delta/\Delta^{m}.
\end{eqnarray*}

Now, let $\hat{\alpha}\in IA^{*}(\hat{\Gamma})$, and write
\[
\hat{\alpha}=(\alpha_{m})_{t|m}\in IA^{*}(\hat{\Gamma})\leq\underset{t|m}{\underleftarrow{\lim}}Aut(\Gamma/\Gamma^{m})
\]
where $\alpha_{m}\in Aut(\Gamma/\Gamma^{m})$. In order to prove the
second part of Theorem \ref{thm:The therorem}, namely that $IA^{*}(\Gamma)$
is dense in $IA^{*}(\hat{\Gamma})$, it is enough to show that:
\begin{prop}
\label{prop:lifting}Let $m$ such that $t|m$, and let $\hat{\alpha},\alpha_{m}$
be as above. Then, there exists $\alpha\in IA^{*}(\Gamma)$ such that
$\alpha\to\alpha_{m}$ through the map $IA^{*}(\Gamma)\to Aut(\Gamma/\Gamma^{m})$.
\end{prop}

So fix $m$ such that $t|m$, and let $\hat{\beta}\in IA^{*}(\hat{\Delta})$
be the image of $\hat{\alpha}$ under the map $IA^{*}(\hat{\Gamma})\to IA^{*}(\hat{\Delta})$.
Write
\[
\hat{\beta}=(\beta_{m})_{t|m}\in IA^{*}(\hat{\Delta})\leq\underset{t|m}{\underleftarrow{\lim}}Aut(\Delta/\Delta^{m})
\]
where $\beta_{m}\in Aut(\Delta/\Delta^{m})$ is the image of $\alpha_{m}\in Aut(\Gamma/\Gamma^{m})$
under the map $Aut(\Gamma/\Gamma^{m})\to Aut(\Delta/\Delta^{m})$.
By assumption, $IA^{*}(\Delta)$ is dense in $IA^{*}(\hat{\Delta})$,
and hence there exists $\beta\in IA^{*}(\Delta)$ such that $\beta\to\beta_{m}$.
So we have the diagram
\[
\begin{array}{ccc}
 &  & \beta\\
 &  & \downarrow\\
\alpha_{m} & \rightarrow & \beta_{m}
\end{array}\qquad\textrm{under\,\,the\,\,diagram}\qquad\begin{array}{ccc}
 &  & IA^{*}(\Delta)\\
 &  & \downarrow\\
Aut(\Gamma/\Gamma^{m}) & \rightarrow & Aut(\Delta/\Delta^{m})
\end{array}
\]
 We want to use $\alpha_{m}$ and $\beta$ in order to construct $\alpha\in IA^{*}(\Gamma)$
such that $\alpha\to\alpha_{m}$. 

Let us recall the following notion: let $P_{1},P_{2}$ and $Q$ be
groups with epimorphisms $\pi_{i}:P_{i}\twoheadrightarrow Q$, $i=1,2$.
The group
\[
U=P_{1}\times_{Q}P_{2}=\{(x,y)\in P_{1}\times P_{2}\,|\,\pi_{1}(x)=\pi_{2}(y)\}
\]
is called the \uline{fiber product} of $P_{1},P_{2}$ along $Q$
(or the \uline{pullback}). It is easy to check that as $\pi_{i}$
are surjective, we get the following commutative diagran of surjective
maps
\[
\begin{array}{ccc}
U & \rightarrow & P_{1}\\
\downarrow &  & \downarrow\\
P_{2} & \rightarrow & Q
\end{array}
\]
such that $P_{1}\simeq U/(U\cap(\{e\}\times P_{2}))$ and $P_{2}\simeq U/(U\cap(P_{1}\times\{e\}))$.
Regarding our context, as we have the commutative surjective diagram
\[
\begin{array}{ccc}
\Gamma & \rightarrow & \Delta\\
\downarrow &  & \downarrow\\
\Gamma/\Gamma^{m} & \rightarrow & \Delta/\Delta^{m}
\end{array}
\]
we get a natural map $\Gamma\overset{\rho}{\to}\Delta\times_{\Delta/\Delta^{m}}\Gamma/\Gamma^{m}$.
Now, as we assume that $t|m$, we have
\[
\ker(\Gamma\to\Gamma/\Gamma^{m})\cap\ker(\Gamma\to\Delta)=\Gamma^{m}\cap tor(\Gamma)=\{e\}
\]
and hence the map $\rho$ is injective. In addition
\[
\ker(\Gamma\to\Delta/\Delta^{m})=\Gamma^{m}\cdot tor(\Gamma)=\ker(\Gamma\to\Gamma/\Gamma^{m})\cdot\ker(\Gamma\to\Delta)
\]
and hence, one can check that it follows that $\rho$ is also surjective.
Therefore, we can identify $\Gamma$ with $\Delta\times_{\Delta/\Delta^{m}}\Gamma/\Gamma^{m}$,
the fiber product of $\Delta$ and $\Gamma/\Gamma^{m}$ along $\Delta/\Delta^{m}$.
The following lemma is elementary:
\begin{lem}
\label{lem:elementary}Notation as above. Let $\sigma_{i}\in Aut(P_{i})$
for $i=1,2$ be automorphisms preserving $\ker(\pi_{i})$ and inducing
$\bar{\sigma}_{i}\in Aut(Q)$. If $\bar{\sigma}_{1}=\bar{\sigma}_{2}$,
then there exists $\sigma\in Aut(U)$ such that $\sigma$ preseves
$U\cap(P_{1}\times\{e\})$ and $U\cap(\{e\}\times P_{2})$ and induces
$\sigma_{1}$ on $P_{1}$ and $\sigma_{2}$ on $P_{2}$.
\end{lem}

\begin{proof}
For $(x,y)\in U$ define $\sigma(x,y)=(\sigma_{1}(x),\sigma_{2}(y))$.
We claim that we have $(\sigma_{1}(x),\sigma_{2}(y))\in H$. Indeed
\[
\pi_{1}(\sigma_{1}(x))=\bar{\sigma}_{1}(\pi_{1}(x))=\bar{\sigma}_{1}(\pi_{2}(y))=\bar{\sigma}_{2}(\pi_{2}(y))=\pi_{2}(\sigma_{2}(y)).
\]
Showing that $\sigma$ is a bijective homomorphism follows from the
same properties for $\sigma_{1},\sigma_{2}$. The other properties
of $\sigma$ follow straightforward from the definition.
\end{proof}
Applying Lemma \ref{lem:elementary} on $\alpha_{m}\in Aut(\Gamma/\Gamma^{m})$
and $\beta\in IA^{*}(\Delta)$ we obtain $\alpha\in Aut(\Gamma)$
that its projection is $\alpha_{m}$ through $Aut(\Gamma)\to Aut(\Gamma/\Gamma^{m})$.
Hence, in order to finish the proof of Theorem \ref{thm:The therorem},
it remains to show that the $\alpha$ we got is not only in $Aut(\Gamma)$
but also in $\alpha\in IA^{*}(\Gamma)$. 
\begin{lem}
\label{lem:iso}Recall that $\Gamma^{*}=\bar{\Gamma}/tor(\bar{\Gamma})$
where $\bar{\Gamma}=\Gamma/\Gamma'$, and denote $\Delta^{*}=\bar{\Delta}/tor(\bar{\Delta})$
where $\bar{\Delta}=\Delta/\Delta'$. Then, we have a canonical isomorphism
$\Gamma^{*}\simeq\Delta^{*}$. 
\end{lem}

\begin{proof}
We have a natural projection
\[
\Delta=\Gamma/tor(\Gamma)\twoheadrightarrow\bar{\Gamma}/tor(\bar{\Gamma})=\Gamma^{*}
\]
which gives rise to a map $\bar{\Delta}=\Delta/\Delta'\twoheadrightarrow\Gamma^{*}$
and to a map $\Delta^{*}=\bar{\Delta}/tor(\bar{\Delta})\twoheadrightarrow\Gamma^{*}$.
Obviously, this map is the inverse of the natural map $\Gamma^{*}\twoheadrightarrow\Delta^{*}$.
So $\Gamma^{*}\simeq\Delta^{*}$ as required.
\end{proof}
From this lemma we get that the preimage of $IA^{*}(\Delta)$ under
the map $Aut(\Gamma)\to Aut(\Delta)$ is $IA^{*}(\Gamma)$. Thus,
as $\alpha$ is a preimage of $\beta$ through $Aut(\Gamma)\to Aut(\Delta)$
and $\beta\in IA^{*}(\Delta)$ it follows that indeed $\alpha\in IA^{*}(\Gamma)$,
as required.

\section{\label{sec:free}The Free cases}

In this section we sketch a more straightforward proof to Corollary
\ref{cor:isomorphism} in the special case where $\Psi_{c}$ is a
free nilpotent group on $n$ elements, a proof which does not refer
neither to the CSP nor to the strong approximation for unipotent groups.
Notice that as $\Psi_{c}^{*}=\Psi_{c}/\Psi_{c}'=\mathbb{Z}^{(n)}$
and $\hat{\Psi}_{c}^{*}=\widehat{\Psi_{c}/\Psi_{c}'}=\mathbb{\hat{Z}}^{(n)}$,
in this case $IA(\Psi_{c})=IA^{*}(\Psi_{c})$ and $IA(\hat{\Psi}_{c})=IA^{*}(\hat{\Psi}_{c})$.
Let us formulate the assertion:
\begin{thm}
Let $\Psi_{c}$ be the free nilpotent group on $n$ elements. For
every $c\in\mathbb{N}$, the map $\widehat{IA(\Psi_{c})}\to IA(\hat{\Psi}_{c})$
is an isomorphism. 
\end{thm}

\begin{proof}
We will prove it by induction on $c$. For $c=1,2$ the result is
trivial, as $\Psi_{1}=\{e\}$ and $\Psi_{2}=\mathbb{Z}^{(n)}$, so
\[
\widehat{IA(\Psi_{1})}=IA(\hat{\Psi}_{1})=\widehat{IA(\Psi_{2})}=IA(\hat{\Psi}_{2})=\{e\}.
\]

For the induction step we will use the (easy) facts that for every
$c$, the natural map $Aut(\Psi_{c+1})\to Aut(\Psi_{c})$ is surjective
(see \cite{key-15}), and the natural map $Aut(\hat{\Psi}_{c+1})\to Aut(\hat{\Psi}_{c})$
is also surjective (see Section 5.2 in \cite{key-9}). So let $c\geq2$.
Denote 
\begin{align*}
A(\Psi_{c+1}) & =\ker(IA(\Psi_{c+1}))\twoheadrightarrow IA(\Psi_{c}))\\
A(\hat{\Psi}_{c+1}) & =\ker(IA(\hat{\Psi}_{c+1}))\twoheadrightarrow IA(\hat{\Psi}_{c})).
\end{align*}
We have the commutative exact diagram
\[
\begin{array}{ccccccccc}
1 & \to & A(\Psi_{c+1}) & \to & IA(\Psi_{c+1}) & \to & IA(\Psi_{c}) & \to & 1\\
 &  & \downarrow &  & \downarrow &  & \downarrow\\
1 & \to & A(\hat{\Psi}_{c+1}) & \to & IA(\hat{\Psi}_{c+1}) & \to & IA(\hat{\Psi}_{c}) & \to & 1
\end{array}
\]
which gives rise to the commutative exact diagram
\[
\begin{array}{cccccccccc}
 &  & \widehat{A(\Psi_{c+1})} & \to & \widehat{IA(\Psi_{c+1})} & \to & \widehat{IA(\Psi_{c})} & \to & 1\\
 &  & \downarrow &  & \downarrow &  & \downarrow\\
1 & \to & A(\hat{\Psi}_{c+1}) & \to & IA(\hat{\Psi}_{c+1}) & \to & IA(\hat{\Psi}_{c}) & \to & 1 & .
\end{array}
\]
By the induction hypothesis and diagram chasing, it is enough to prove
that the map $\widehat{A(\Psi_{c+1})}\to A(\hat{\Psi}_{c+1})$ is
an isomorphism. 

Let $x_{1},\ldots,x_{n}$ be a set of free generators for $\Psi_{c+1}$.
Denote 
\[
Z(\Psi_{c+1})^{(n)}=\underset{n}{\underbrace{Z(\Psi_{c+1})\times...\times Z(\Psi_{c+1})}}.
\]
where $Z(\Psi_{c+1})$ is the center of $\Psi_{c+1}$. Observe now
that $A(\Psi_{c+1})$ can be viewed as a subgroup of $Z(\Psi_{c+1})^{(n)}$
in the following way: For $\alpha\in A(\Psi_{c+1})$, one can describe
it by its action on the generators of $\Psi_{c+1}$ and write $\alpha(x_{i})=x_{i}u_{i}$
for some $u_{i}\in Z(\Psi_{c+1})$. We claim that the map $\alpha\mapsto(u_{1},\ldots,u_{n})\in Z(\Psi_{c+1})^{(n)}$
is a natural injective \uline{homomorphism} $A(\Psi_{c+1})\to Z(\Psi_{c+1})^{(n)}$.
Indeed, let $\alpha,\beta\in A(\Psi_{c+1})$ defined by $\alpha(x_{i})=x_{i}u_{i}$
and $\beta(x_{i})=x_{i}v_{i}$ for some $u_{i},v_{i}\in Z(\Psi_{c+1})$.
Now, as $u_{i}\in Z(\Psi_{c+1})\subseteq\Psi_{c+1}'$ it can be written
as a product of commutators of words on the generators $x_{i}$. As
$u_{i},v_{i}\in Z(\Psi_{c+1})$, it is easy to see that $(\beta\circ\alpha)(x_{i})=x_{i}v_{i}u_{i}$.
This proves the claim. In \cite{key-15} it is proven that the above
map $A(\Psi_{c+1})\to Z(\Psi_{c+1})^{(n)}$ is also surjective, and
hence, it is an isomorphism. I.e., for every choice of elements $u_{1},\ldots,u_{n}\in Z(\Psi_{c+1})$,
one can define an element $\alpha\in A(\Psi_{c+1})$ by defining $\alpha(x_{i})=x_{i}u_{i}$
for the generators $x_{1},\ldots,x_{n}$ of $\Psi_{c+1}$. This is
indeed an automorphism as $\{x_{i}u_{i}\}_{i=1}^{n}$ generate $\Psi_{c+1}$,
which is a free nilpotent group.

Using a similar approach, one can prove that $A(\hat{\Psi}_{c+1})\cong Z(\hat{\Psi}_{c+1})^{(n)}$
(see also \cite{key-9}) and thus 
\[
\widehat{A(\Psi_{c+1})}\cong\widehat{Z(\Psi_{c+1})}^{(n)}\cong Z(\hat{\Psi}_{c+1})^{(n)}\cong A(\hat{\Psi}_{c+1})
\]
as required.
\end{proof}

\begin{tabular}{>{\raggedright}p{6cm}l}
Ben-Ezra, David El-Chai & Lubotzky, Alexander\tabularnewline
Department of Mathematics & Einstein Institute of Mathematics\tabularnewline
University of California in San-Diego & The Hebrew University\tabularnewline
La Jolla, CA 92093 & Jerusalem, 91904\tabularnewline
USA  & Israel\tabularnewline
davidel-chai.ben-ezra@mail.huji.ac.il & alex.lubotzky@mail.huji.ac.il\tabularnewline
\end{tabular}
\end{document}